\documentclass[12pt,sn-mathphys-ay]{sn-jnl}

\usepackage{graphicx}%
\usepackage{multirow}%
\usepackage{amsmath,amssymb,amsfonts}%
\usepackage{amsthm}%
\usepackage{mathrsfs}%
\usepackage[title]{appendix}%
\usepackage{xcolor}%
\usepackage{textcomp}%
\usepackage{manyfoot}%
\usepackage{booktabs}%
\usepackage{algorithm}%
\usepackage{algorithmicx}%
\usepackage{algpseudocode}%
\usepackage{listings}%
\usepackage{hyperref}
\usepackage{dsfont}
\theoremstyle{thmstyleone}%
\newtheorem{theorem}{Theorem}
\newtheorem{theoremletter}{Theorem}

\newtheorem{lemmaletter}[theoremletter]{Lemma}

\theoremstyle{thmstyletwo}%
\newtheorem{remark}{Remark}%

\theoremstyle{thmstylethree}%
\begin{document}
\newcommand{\R}{\mathds{R}}
\newcommand{\M}{M^{+}(\R^n)}
\newcommand{\wa}{\mathbf{W}_{\alpha,p}}
\newcommand{\wao}{\mathbf{W}_{\alpha,p}}
\newcommand{\wat}{\mathbf{W}_{\alpha,p}}
\newcommand{\wak}{\mathbf{W}_{k}}
\newcommand{\wah}{\mathbf{W}_{\frac{2k}{k+1},k+1}}
\newcommand{\ia}{\mathbf{I}_{2\alpha}}
\newcommand{\Om}{\Omega}
\newcommand{\e}{\epsilon}
\newcommand{\phik}{\Phi^{k}(\Om)}
\newcommand{\la}{\lambda}

\title{On a fully nonlinear k-Hessian system of Lane-Emden type}
\author{\fnm{Genival} \sur{da Silva}\footnote{email: gdasilva@tamusa.edu, website: \url{www.gdasilvajr.com}}}
\affil{\orgdiv{Department of Mathematics}, \orgname{Texas A\&M University - San Antonio}}


\abstract{In this manuscript we prove the existence of solutions to a fully nonlinear system of (degenerate) elliptic equations of Lane-Emden type and discuss a inhomogeneous generalization.}




\maketitle
\section{Introduction}\label{sec1}
We will study the following fully nonlinear (degenerate) Elliptic system:
\begin{equation}\label{FNS}\tag{$S$}
     \left\{
\begin{aligned}
& F_{k}[-u]= \sigma\, v^{q_1}, \quad v>0 \quad \mbox{in}\quad \mathbb{R}^n,\\ 
& F_{k}[-v]= \sigma\, u^{q_2}, \quad u>0 \quad \mbox{in}\quad \mathbb{R}^n,\\
& \liminf\limits_{|x|\to \infty}u(x)=0, \quad \liminf\limits_{|x|\to \infty}v(x)=0.
\end{aligned}
\right.
 \end{equation}
 where $\sigma\in M^{+}(\mathds{R}^{n})$ is a nonnegative Radon measure, $ \, 0<q_i<k<\frac{n}{2}$ for $ i=1,\, 2$, and $F_{k}[u]$ is the k-Hessian operator, defined as the sum of the k-minors of $D^{2}u$.
 
 The equation $F_{k}[u]=f(x,u)$ in a domain $\Om\subseteq \R^{n}$ actually is variational. It's the Euler-Lagrange equation of the functional
 $$I_{k}[u]=\frac{1}{k+1}\int_{\Om} u_{i} u_{j} F^{ij}_{k}[u]-\int_{\Om} F(x,u),$$
where $F^{ij}_{k}$ represents the derivative with respect to the entry $a_{ij}$ and $F(x,u)$ an antiderivative with respect to $u$. 

Similarly to the linear case, when we study the variational problem associated with the k-Hessian equation we divide the problem into three cases, the sublinear case, the eigenvalue problem, and the superlinear case, depending on the behavior of $f(x,u)$. In this manuscript we restrict ourselves to the case of systems with $f(x,u)=u^{q}$, with $q<k$. However, the results presented here could be generalized to other types of $f(x,u)$ having similar growth rate.

 In order for this system \eqref{FNS} to make sense in case $u$ is not regular enough, we consider it in the sense of k-subharmonic functions and k-Hessian measures as defined in \cite{wang1999}. 
 
 More precisely, an upper-semicontinuous function is said to be \textit{k-subharmonic} in $\Om\subseteq \R^{n}$, $1\leq k \leq n$, if $F_{k}[q]\geq 0$ for any quadratic polynomial $q(x)$ such that $u(x)-q(x)$ has a local finite maximum in $\Om$. In case, $u\in C^{2}_{loc}(\Om)$ then $u$ is k-subharmonic if and only if $F_{j}[u]\geq 0$ for $j\leq k$. In particular, all k-subharmonic functions are subharmonic in the usual sense and if $k=n$ then they are also convex. 
 
 The set of all k-subharmonic functions in $\Om$ will be denoted by $\phik$.
 
 We recall the following theorem:
 \begin{theorem}[\cite{wang1999}]\label{thm1}
For any k-subharmonic function u, there exists a Radon measure $\mu_{k}[u]$, called \textit{the k-Hessian measure} such that
\begin{enumerate}
	\item[i.] If $u\in C^{2}(\Om)$ then $\mu_{k}[u]=F_{k}[u]dx$; and
	\item[ii.] If $u_{j}$ is a sequence of k-subharmonic functions which converges to $u$ a.e., then $\mu_{k}[u_{j}]\to \mu_{k}[u]$ weakly as measures.
\end{enumerate}
 \end{theorem}
 
 Hence, system \eqref{FNS} makes sense even in case $u,v$ are not $C^{2}$ but are k-subharmonic, as long as we understand $F_{k}[u]$ as $\mu_{k}[u]$.
 
 The case $k=1$ of system \eqref{FNS} , i.e.  $F_{1}[u]=\Delta u$, and its $p$-laplacian generalization, have been addressed recently in \cite{bib1}, where the authors obtained Brezis-Kamin type estimates \cite{brezis92} for a quasi-linear system similar to the one being discussed here. In \cite{dasilva24}, we analyzed the Fractional Laplacian equivalent of \eqref{FNS}. We plan to generalize this circle of ideas and discuss the fully nonlinear case when $1<k<\frac{n}{2}$. 
 
 For more on the k-Hessian equation see the wonderful monograph \cite{wang2009}.
 
 System \eqref{FNS} above is related to the following system
 \begin{equation}\label{SW}\tag{$W$}
     \left\{
\begin{aligned}
& u= \wah\left(v^{q_1}\mathrm{d} \sigma\right), \quad \mathrm{d}\sigma \mbox{-a.e in } \mathds{R}^n,\\
& v= \wah\left(u^{q_2}\mathrm{d} \sigma\right), \quad \mathrm{d}\sigma \mbox{-a.e in } \mathds{R}^n.
\end{aligned}
\right.
 \end{equation}
Where $ \wa\mu$ is the Wolff Potential, defined by
\begin{equation}\label{wolff potential}
    \wa\mu(x)=\int_0^{\infty} \left(\frac{\mu(B(x,t))}{t^{n-\alpha p}}\right)^{\frac{1}{p-1}}\frac{\mathrm{d} t}{t}, \quad x\in\mathds{R}^n.
\end{equation}
and consequently,
\begin{equation}\label{wolff potential}
    \wah\mu(x)=\int_0^{\infty} \left(\frac{\mu(B(x,t))}{t^{n-2k}}\right)^{\frac{1}{k}}\frac{\mathrm{d} t}{t}, \quad x\in\mathds{R}^n.
\end{equation}

 \noindent\textbf{Notation:} For the sake of simplicity will denote $\wah\sigma(x)$ by $\wak\sigma(x)$.

In this paper we will only consider measures $\mu\in  {M}^+(\mathds{R}^n)$ who satisfy
\begin{equation}\label{FIN}\tag{FIN}
    \wa\mu(x) < \infty \;\text{ a.e. in } \mathds{R}^n
\end{equation}
The following theorem is an adaptation, for our purposes, of \cite[Thm.~1.1]{bib1}.
\begin{theorem}\label{doo}
Let $1\leq k < \frac{n}{2}$ and consider $\sigma\in  {M}^+(\mathds{R}^n)$ satisfying
\begin{equation}\label{cap}\tag{C}
           \sigma(E)\leq C_{\sigma}\,\mathrm{cap}_{\frac{2k}{k+1},k+1}(E) \quad \mbox{for all compact sets } E\subset \mathds{R}^n,
\end{equation}
and \eqref{FIN}, where $\mathrm{cap}_{\alpha,p}(E)$ is the $(\alpha,p)$-capacity defined by
\begin{equation}\label{capa}
    \mathrm{cap}_{\alpha,p}(E)=\inf \{\|f\|_{L^{p}}^p: 
    f \in L^{p}(\mathds{R}^n), \;  f \geq 0, \; \mathbf{I}_{\alpha}f\geq 1\ \mbox{on}\ E\}.
\end{equation}
Then there exists a solution $(u,v)$ to system \eqref{SW} such that 
\begin{equation}\label{est1}
    \begin{aligned}
    & c^{-1}\left(\wak\sigma\right)^{\frac{k(k+q_1)}{k^2-q_1q_2}}\leq u\leq c\left(\wak\sigma + \left(\wak\sigma\right)^{\frac{k(k+q_1)}{k^2-q_1q_2}} \right),\\
   & c^{-1}\left(\wak\sigma\right)^{\frac{k(k+q_2)}{k^2-q_1q_2}}\leq v\leq c\left(\wak\sigma + \left(\wak\sigma\right)^{\frac{k(k+q_2)}{k^2-q_1q_2}} \right),
    \end{aligned}
\end{equation}
where $c=c(n,p,q_1,q_2,\alpha,C_{\sigma})>0$. Furthermore, $u,v \in L_{\mathrm{loc}}^{s}(\mathds{R}^n,\, \mathrm{d} \sigma)$, for every $s>0$.
\end{theorem}
Our main result is the following theorem which is a direct application of the theorem above:
\begin{theorem}\label{maint}
Let $1< k < \frac{n}{2}$ and consider $\sigma\in  {M}^+(\mathds{R}^n)$ satisfying \eqref{cap} and \eqref{FIN}. Then there exists a solution $(u,v)\in \Phi^{k}(\R^{n})\times \Phi^{k}(\R^{n})$ to Syst. \eqref{FNS} such that 
\begin{equation}\label{est1}
    \begin{aligned}
    & c^{-1}\left(\wak\sigma\right)^{\frac{k(k+q_1)}{k^2-q_1q_2}}\leq u\leq c\left(\wak\sigma + \left(\wak\sigma\right)^{\frac{k(k+q_1)}{k^2-q_1q_2}} \right),\\
   & c^{-1}\left(\wak\sigma\right)^{\frac{k(k+q_2)}{k^2-q_1q_2}}\leq v\leq c\left(\wak\sigma + \left(\wak\sigma\right)^{\frac{k(k+q_2)}{k^2-q_1q_2}} \right),
    \end{aligned}
\end{equation}
where $c=c(n,p,q_1,q_2,\alpha,C_{\sigma})>0$. Furthermore, $u,v \in L_{\mathrm{loc}}^{s}(\mathds{R}^n,\, \mathrm{d} \sigma)$, for every $s>0$.
\end{theorem}
In the next theorem we provide a fully nonlinear equivalent result similar to what is presented in \cite{verb17}, \cite{bib1}. 
\begin{theorem}\label{sect}
Let $1< k < \frac{n}{2}$ and consider $\sigma\in  {M}^+(\mathds{R}^n)$ satisfying \eqref{FIN}. Then there exists a solution $(u,v)\in \Phi^{k}(\R^{n})\times \Phi^{k}(\R^{n})$ to System \eqref{FNS} satisfying \eqref{est1} if and only if there exists $\lambda>0$ such that almost everywhere we have
\begin{equation}\label{est2}
    \begin{aligned}
  & \wak\left((\wak\sigma)^{q_{1}\gamma_{2}} \right)\leq \lambda\left(\wak\sigma + \left(\wak\sigma\right)^{\gamma_{1}} \right) <\infty, 
  \\
   & \wak\left((\wak\sigma)^{q_{2}\gamma_{1}} \right)\leq \lambda\left(\wak\sigma + \left(\wak\sigma\right)^{\gamma_{2}} \right) <\infty .
    \end{aligned}
\end{equation}
\end{theorem}
In section \ref{INC} we analyze a generalization of System \eqref{SW}, we consider
\begin{equation}\label{IW}\tag{$IW$}
     \left\{
\begin{aligned}
& u= \wa\left(v^{q_1}\mathrm{d} \sigma + \mathrm{d} \mu\right), \quad \mathrm{d}\sigma \mbox{-a.e in } \mathds{R}^n,\\
& v= \wa\left(u^{q_2}\mathrm{d} \sigma + \mathrm{d} \nu\right), \quad \mathrm{d}\sigma \mbox{-a.e in } \mathds{R}^n, 
\end{aligned}
\right.
 \end{equation}
 where $\mu,\nu\in M^{+}(\R^{n})$ are given. We aim to prove the following result.
 \begin{theorem}\label{gen}
Let $\sigma,\mu,\nu \in  {M}^+(\mathds{R}^n)$ satisfy \eqref{cap} and \eqref{FIN}
Then there exists a solution $(u,v)$ to system \eqref{IW} such that 
\begin{equation}\label{est3}
    \begin{aligned}
    & c^{-1}\left(\wa\sigma\right)^{\gamma_{1}}\leq u\leq c\left(\wa\mu+\wa\nu+\wa\sigma + \left(\wa\sigma\right)^{\gamma_{1}} \right),\\
   & c^{-1}\left(\wa\sigma\right)^{\gamma_{2}}\leq v\leq c\left(\wa\mu+\wa\nu+\wa\sigma + \left(\wa\sigma\right)^{\gamma_{2}} \right),
    \end{aligned}
\end{equation}
where $c=c(n,p,q_1,q_2,\alpha,C_{\sigma})>0$. Furthermore, $$u,v \in L_{\mathrm{loc}}^{s}(\mathds{R}^n,\, \mathrm{d} \mu)+L_{\mathrm{loc}}^{s}(\mathds{R}^n,\, \mathrm{d} \nu)+L_{\mathrm{loc}}^{s}(\mathds{R}^n,\, \mathrm{d} \sigma),$$ for every $s>0$.
\end{theorem}
\subsection*{Organization of the paper} 
In Secion~\ref{mainS} we give some preliminaries and present the proof of theorems \ref{maint} and \ref{sect}. In Section~\ref{INC},  we discuss generalizations and prove theorem \ref{gen}. In Section~\ref{last}, we present our final remarks and some open problems.
\subsection*{Notations and definitions}  
We assume $\Omega\subseteq \R^n$ is a domain. We denote by $M^+(\Omega)$ the space of all nonnegative locally finite Borel measures on $\Omega$ and $\sigma(E)= \int_E d \sigma$ the $\sigma$-measure of a measurable set $E\subseteq \Omega$. The letter $c$ or $C$ will always denote a positive constant which may vary from line to line. We understand  $F_{k}[u]$ as the k-Hessian measure $ \mu_{k}[u]$ associated to $u$, in particular we do not assume $u\in C^{2}$.
\section{Main results}\label{mainS}
We will need the following result from \cite{verb2008}:
\begin{lemmaletter}\label{lemA}
Let $u\geq 0$ be such that $-u\in\Phi_{k}(\R^{n})$, where $1\leq k<\frac{n}{2}$. If $\mu=\mu_{k}[-u]$ and $\inf_{\R^{n}} u=0$. Then for all $x\in\R^{n}$,
\begin{equation}
K^{-1}\wak\mu (x)\leq u(x) \leq K \wak\mu (x) ,
\end{equation}
for a constant $K$ depending only on $n$ and $k$.
\end{lemmaletter}
The following lemma from \cite{verb17} will be needed for the proof of the main theorem.
\begin{lemmaletter}\label{lemB}
    Let $\omega\in M^+(\mathds{R}^n)$. For every $r>0$ and for all $x\in \mathds{R}^n$, it holds
    \begin{equation}
        \wa\left((\wa\omega)^r\mathrm{d} \omega\right)(x)\geq \kappa^{\frac{r}{p-1}}\left(\wa\omega(x)\right)^{\frac{r}{p-1}+1},
    \end{equation}
    where $\kappa$ depends only on $n, p$.
\end{lemmaletter}
\subsection*{Proof of Theorem \ref{maint}}
\begin{proof}
The proof is similar to the arguments presented in \cite{verb2008},\cite{verb17}, \cite{bib1} and uses the idea of successive approximations.
Let $(\overline{u},\overline{v})$ be the solution to the system \eqref{SW} given by theorem \ref{doo}. By hypothesis, $\overline{u}(x),\overline{v}(x)$ satisfy estimate \eqref{est1} , which implies $$\liminf\limits_{|x|\to \infty}\overline{u}(x)=0, \quad \liminf\limits_{|x|\to \infty}\overline{v}(x)=0.$$ Also, notice that $\overline{v}^{q_{1}}d\sigma,\overline{u}^{q_{2}}d\sigma\in L^{1}_{loc}$. This will be important for the uniqueness below.
Set
\begin{equation*}
     \gamma_1=\frac{k(k+q_1)}{k^2-q_1q_2}\quad \mbox{and} \quad
       \gamma_2=\frac{k(k+q_2)}{k^2-q_1q_2},
\end{equation*}
and notice that
\begin{equation*}
    \gamma_1=\frac{q_1}{k}{\gamma_2}+1 \quad \mbox{and}\quad {\gamma_2}=\frac{q_2}{k}\gamma_1+1.
\end{equation*}
Now, define $u_{0}(x)=\la \left(\wak\sigma\right)^{\gamma_{1}}$ and $v_{0}(x)=\la \left(\wak\sigma\right)^{\gamma_{2}}$, where $\la$ is a constant to be chosen. Using lemma \ref{lemB}, we obtain:
$$\wak (v_{0}^{q_{1}}d\sigma)=\wak ((\la \left(\wak\sigma\right)^{\gamma_{2}})^{q_{1}}d\sigma)\geq C\, \la \left(\wak\sigma\right)^{\gamma_{1}}.$$
By choosing $\la$ suitably small it's possible to obtain $u_0\leq \wak(v_0^{q_1}\mathrm{d}\sigma),v_0\leq\wak(u_0^{q_2}\mathrm{d}\sigma)\leq \overline{v}$ and $$u_{0}\leq \overline{u},\;v_{0}\leq \overline{v}.$$
For $i=1,2,\ldots$, let $B_{i}$ denote the ball centered at the origin of radius $i$ and consider the Dirichlet problem:
\begin{equation}\label{DS}\tag{D}
     \left\{
\begin{aligned}
& F_{k}[-u_{1}^{i}]= \sigma\, v_{0}^{q_1}  \;\text{in}\; B_{i},\\\ 
& u_{1}^{i}=0 \;\text{on}\; \partial B_{i},\\
\end{aligned}
\right.
 \end{equation}
 There is of course an equivalent system associated to $v(x)$, namely, functions $v_{1}^{i}$ defined in $B_{i}$.
 
 By the existence result \cite[Thm.~8.2]{wang2009} and comparison principle \cite[Thm.~4.1]{wang2002}, system \eqref{DS} above has a unique solution. 
 \begin{remark}
 Notice that we only have uniqueness because $\sigma\, v_{0}^{q_1}, \sigma\, u_{0}^{q_2}\in L^{1}$, which implies in particular continuity with respect to the k-Hessian capacity. It's not known if uniqueness is valid for general $\mu$ not integrable.
 \end{remark}
 Moreover, the sequence of functions $u_{1}^{i}(x)$ is increasing, i.e. $u_{1}^{i}(x)\leq u_{1}^{i+1}(x)$, and we claim it is also bounded above. Indeed, by \cite[Thm.~7.2]{verb2008}, the following estimate is true for every $i$: $$u_{1}^{i}(x)\leq C\wak v_{0}^{q_1}\leq  C\wak \overline{v}^{q_1}=\overline{u}$$
 Therefore, we conclude that there are k-subharmonic functions $u_{1}(x),v_{1}(x)$ such that $$u_{1}=\lim_{i\to\infty} u_{1}^{i} ,\;\; v_{1}=\lim_{i\to\infty} v_{1}^{i}$$
 By the weak continuity of the k-Hessian measure, namely theorem \ref{thm1}, we have $(u_{1},v_{1})$ satisfy the system
 \begin{equation}
     \left\{
\begin{aligned}
& F_{k}[-u_{1}]= \sigma\, v_{0}^{q_1}  \;\text{in}\; \R^{n},\\\ 
& F_{k}[-v_{1}]= \sigma\, u_{0}^{q_2}  \;\text{in}\; \R^{n}.\\
\end{aligned}
\right.
 \end{equation}
 Also, notice that $u_{1}\leq \overline{u},\; v_{1}\leq \overline{v}$, so we automatically have 
 \begin{equation*}
    \liminf\limits_{|x|\to \infty}u_{1}(x)=\liminf\limits_{|x|\to \infty}v_{1}(x)=0.
\end{equation*}
Additionally,  it follows that 
 \begin{equation}
\begin{aligned}
& u_{1}\geq K^{-1} \wak( v_{0}^{q_1}d\sigma)\geq u_{0},\\\ 
& v_{1}\geq K^{-1} \wak( u_{0}^{q_2}d\sigma)\geq v_{0}.\\
\end{aligned}
 \end{equation}
 We apply induction to obtain sequences of functions $u_{i},v_{i}$ such that 
 \begin{equation}
     \left\{
\begin{aligned}
& F_{k}[-u_{i}]=\sigma\, v_{i-1}^{q_1}\quad \mbox{in}\quad \mathds{R}^n,\\ 
&  F_{k}[-v_{i}]=\sigma\, u_{i-1}^{q_2}\quad \mbox{in}\quad \mathds{R}^n,\\
& u_i\leq \overline{u},\ \ v_i\leq \overline{v} \quad \mbox{in}\quad \mathds{R}^n,\\
& 0\leq u_{i-1}\leq u_j, \ \ 0\leq v_{i-1}\leq v_i, \\
& \liminf\limits_{|x|\to \infty}u_{i}(x)=\liminf\limits_{|x|\to \infty}v_{i}(x)=0,
\end{aligned}
\right.
\end{equation}
By the monotone convergence theorem and the weak continuity of the k-Hessian, we conclude that there are functions $u,v$ satisfying
\begin{equation}
     \left\{
\begin{aligned}
& F_{k}[-u]=\sigma\, v^{q_1}\quad \mbox{in}\quad \mathds{R}^n,\\ 
&  F_{k}[-v]=\sigma\, u^{q_2}\quad \mbox{in}\quad \mathds{R}^n,\\
& \liminf\limits_{|x|\to \infty}u(x)=\liminf\limits_{|x|\to \infty}v(x)=0,
\end{aligned}
\right.
\end{equation}
The remaining statements in the theorem follow directly from the fact that $$u_{0}\leq u\leq \overline{u},\;\; v_{0}\leq v\leq \overline{v}.$$ and  \cite[Lem.~3.4]{bib1}(alternatively, lemma \ref{lemD} below).
\end{proof}
\subsection*{Proof of Theorem \ref{sect}}
\begin{proof}\;
\begin{enumerate}
	\item[$(\Rightarrow)$] Suppose System \eqref{FNS} has a solution $(u,v)\in \Phi^{k}(\R^{n})\times \Phi^{k}(\R^{n})$ satisfying \eqref{est1}, then 
\begin{equation}
    \begin{aligned}
    &c^{-1}\left(\wak\sigma\right)^{\gamma_{1}}  \leq  u\leq c\left(\wak\sigma + \left(\wak\sigma\right)^{\gamma_{1}} \right),\\
   & c^{-1}\left(\wak\sigma\right)^{\gamma_{2}}  \leq  v\leq c\left(\wak\sigma + \left(\wak\sigma\right)^{\gamma_{2}} \right),
    \end{aligned}
\end{equation}
for some $c>0$. Now, by lemma \ref{lemA}, 
\begin{equation}
\begin{aligned}
    & K^{-1}\wak(v^{q_{1}}d\sigma)\leq u(x) ,\\
   & K^{-1}\wak(u^{q_{2}}d\sigma)\leq v(x),
    \end{aligned}
\end{equation}
therefore we obtain:
\begin{equation}
\begin{aligned}
    & K^{-1}\wak(v^{q_{1}}d\sigma)\geq C\,\mathbf{W}_{k}\left((\mathbf{W}_{k}\sigma)^{q_{1}\gamma_{2}}\mathrm{d} \sigma\right),\\
   & K^{-1}\wak(u^{q_{2}}d\sigma)\geq C\,\mathbf{W}_{k}\left((\mathbf{W}_{k}\sigma)^{q_{2}\gamma_{1}}\mathrm{d} \sigma\right).
    \end{aligned}
    \end{equation}
 If we set $\lambda=\frac{c}{C}>0$ the result follows.
 \item[$(\Leftarrow)$] This is a direct consequence of \cite[Thm.~1.3]{bib1}, since it give us a solution $(\overline{u},\overline{v})$ to the system \eqref{SW}. If we proceed mutatis mutandis like in the proof of theorem \ref{maint} we obtain the desired solution.
\end{enumerate}
\end{proof}
\begin{remark}
In the linear case, $k=1, u=v, q_{1}=q_{2}=q$ and $\sigma$ radially symmetric, we have recently \cite{dasilva24} given a existence criteria for a weaker criteria than the one above. In particular, it's possible to have solutions that do not satisfy the estimates \eqref{est1}. More precisely, the following example given in \cite{verb16}, gives a counter-example:
\begin{equation}
\sigma(y)= \left\{ \begin{array}{ll}
\frac{1}{|y|^s \log^{\beta}\frac{1}{|y|}} \, , \quad \text{ if } |y| < 1/2,\\
0 , \quad \text{ if } |y| \ge 1/2,
\end{array}\right.
\end{equation} 
where $s=(1-q)n+2 q$ and $\beta > 1$.
It is unknown at the time of the writing of this paper if it's possible to have solutions without \eqref{est1} holding when $k>1$.
\end{remark}
\section{The inhomogeneous case}\label{INC}
In this section we study the case where the complexity of the system is increased due to measures $\mu,\nu\in M^{+}(\R^{n})$. We consider
\begin{equation}\label{inS}\tag{$INH$}
     \left\{
\begin{aligned}
& F_{k}[-u]= \sigma\, v^{q_1} + \mu, \quad v>0 \quad \mbox{in}\quad \mathbb{R}^n,\\ 
& F_{k}[-v]= \sigma\, u^{q_2} + \nu, \quad u>0 \quad \mbox{in}\quad \mathbb{R}^n,\\
& \liminf\limits_{|x|\to \infty}u(x)=0, \quad \liminf\limits_{|x|\to \infty}v(x)=0.
\end{aligned}
\right.
 \end{equation}
 As before, this system is related to the more general integral system
  \begin{equation}\label{IW}\tag{$IW$}
     \left\{
\begin{aligned}
& u= \wa\left(v^{q_1}\mathrm{d} \sigma + \mathrm{d} \mu\right), \quad \mathrm{d}\sigma \mbox{-a.e in } \mathds{R}^n,\\
& v= \wa\left(u^{q_2}\mathrm{d} \sigma + \mathrm{d} \nu\right), \quad \mathrm{d}\sigma \mbox{-a.e in } \mathds{R}^n.
\end{aligned}
\right.
 \end{equation}
 We claim that our results generalize to this scenario as well. First, we need a generalization of lemma \ref{lemB}:
\begin{lemmaletter}\label{lemC}
    Let $\sigma, \mu\in M^+(\mathds{R}^n)$. For every $r>0$ and for all $x\in \mathds{R}^n$, it holds
    \begin{equation}
        \wa\left((\wa\sigma)^r\mathrm{d} \sigma + \mathrm{d} \mu \right)(x)\geq \kappa^{\frac{r}{p-1}}\left(\wa\sigma)\right)^{\frac{r}{p-1}+1},
    \end{equation}
    where $\kappa$ depends only on $n, p$.
\end{lemmaletter}
\begin{proof} The proof is immediate since $(\wa\sigma)^r\mathrm{d} \sigma + \mathrm{d} \mu \geq (\wa\sigma)^r\mathrm{d} \sigma$.
\end{proof}
Like before, the lemma above is the key in showing that a subsolution exists. We recall the following result from \cite[Lem.~3.4]{bib1}:
\begin{lemmaletter}\label{lemD}
    Let $\sigma\in M^+(\mathds{R}^n)$ satisfies \eqref{cap} and \eqref{FIN}. Then 
    \begin{equation}
\int_{B(x,R)}(\mathbf{W}_{\alpha,p}\sigma)^s\,\mathrm{d}\sigma \leq c(\sigma(B(x,2R)+\sigma(B(x,R))\\
\end{equation}
\end{lemmaletter}
\subsection*{Proof of Theorem \ref{gen}}
\begin{proof}
The method of proof is similar to \cite{verb17,bib1,dasilva24}, namely use sub-super solutions. Even though it's not entirely obvious that the same method should work due to the interactions between $\sigma,\mu,\nu$, as pointed out by \cite[Sec.~5]{bib1}, it turns out that with some adjustments the proof is still valid. We recall the following notation:
\begin{equation*}
     \gamma_1=\frac{(p-1)(p-1+q_1)}{(p-1)^2-q_1q_2} \quad \mbox{and} \quad
       \gamma_2=\frac{(p-1)(p-1+q_2)}{(p-1)^2-q_1q_2},
\end{equation*}
Notice that by definition
\begin{equation*}
    \gamma_1=\frac{q_1}{p-1}{\gamma_2}+1 \quad \mbox{and}\quad {\gamma_2}=\frac{q_2}{p-1}\gamma_1+1.
\end{equation*}
The idea to find a subsolution is to use lemma \ref{lemC}. Fix $\lambda>0$, to be defined later and define 
\begin{equation*}
    (\underline{u},\underline{v})= \big(\lambda (\mathbf{W}_{\alpha,p}\sigma)^{\gamma_1}, \lambda (\mathbf{W}_{\alpha,p}\sigma)^{{\gamma_2}}\big)
\end{equation*}
Then we obtain
\begin{align*}
     \wao(\underline{v}^{q_1}\mathrm{d}\sigma+\mathrm{d}\mu)&= {\lambda}^{\frac{q_1}{p-1}}\wao( (\mathbf{W}_{\alpha,p}\sigma)^{q_1{\gamma_2}}\mathrm{d}\sigma+\mathrm{d}\mu)\geq {\lambda}^{\frac{q_1}{p-1}}C(\mathbf{W}_{\alpha,p}\sigma)^{\frac{q_1}{p-1}{\gamma_2}+1} \\
    & = {\lambda}^{\frac{q_1}{p-1}}C(\mathbf{W}_{\alpha,p}\sigma)^{\gamma_1},\\    
    \wat(\underline{u}^{q_2}\mathrm{d}\sigma+\mathrm{d}\nu)&= {\lambda}^{\frac{q_2}{p-1}}\wat( (\mathbf{W}_{\alpha,p}\sigma)^{q_2{\gamma_1}}\mathrm{d}\sigma+\mathrm{d}\nu)\geq {\lambda}^{\frac{q_2}{p-1}}C(\mathbf{W}_{\alpha,p}\sigma)^{\frac{q_2}{p-1}{\gamma_1}+1} \\
    & = {\lambda}^{\frac{q_2}{p-1}}C(\mathbf{W}_{\alpha,p}\sigma)^{\gamma_2}.
\end{align*}
By choosing $\lambda$ sufficiently small, we can guarantee that
\begin{equation*}
\begin{aligned}
    & \underline{u}\leq \wao(\underline{v}^{q_1}\mathrm{d}\sigma +\mathrm{d}\mu),\\
    & \underline{v}\leq \wat(\underline{u}^{q_2}\mathrm{d}\sigma +\mathrm{d}\nu),
\end{aligned}
\end{equation*}
which confirms that $(\underline{u},\underline{v})$ is a subsolution. 

The goal now is to find a supersolution. Let's now define 
\begin{equation*}
(\overline{u},\overline{v})=\big(\lambda\left(\mathbf{W}_{\alpha,p}\mu +\mathbf{W}_{\alpha,p}\nu +\mathbf{W}_{\alpha,p}\sigma + \left(\mathbf{W}_{\alpha,p}\sigma\right)^{\gamma_1} \right), \lambda\left(\mathbf{W}_{\alpha,p}\mu +\mathbf{W}_{\alpha,p}\nu +\mathbf{W}_{\alpha,p}\sigma + \left(\mathbf{W}_{\alpha,p}\sigma\right)^{{\gamma_2}} \right)\big)
\end{equation*}
where $\lambda$ is a constant to be determined later. Then we have
\begin{align}
    \wao(\overline{v}^{q_1}\mathrm{d}\sigma+\mathrm{d}\mu)(x)&\leq {\lambda}^{\frac{q_1}{p-1}}C\int_0^{\infty}\Big(\frac{\int_{B(x,t)}(\mathbf{W}_{\alpha,p}\sigma)^{q_1}+(\mathbf{W}_{\alpha,p}\sigma)^{{\gamma_2}q_1}\mathrm{d}\sigma + \mu(B(x,t))}{t^{n-\alpha p}}\Big)^{\frac{1}{p-1}}\frac{\mathrm{d} t}{t} \nonumber\\
    &\leq {\lambda}^{\frac{q_1}{p-1}}C\Big[\int_0^{\infty}\Big(\frac{\int_{B(x,t)}(\mathbf{W}_{\alpha,p}\mu)^{q_1}\mathrm{d}\sigma}{t^{n-\alpha p}}\Big)^{\frac{1}{p-1}}\frac{\mathrm{d} t}{t} \nonumber\\
    &\Big.+\int_0^{\infty}\Big(\frac{\int_{B(x,t)}(\mathbf{W}_{\alpha,p}\nu)^{q_1}\mathrm{d}\sigma}{t^{n-\alpha p}}\Big)^{\frac{1}{p-1}}\frac{\mathrm{d} t}{t}\nonumber\\
    &\Big.+\int_0^{\infty}\Big(\frac{\int_{B(x,t)}(\mathbf{W}_{\alpha,p}\sigma)^{q_1}\mathrm{d}\sigma}{t^{n-\alpha p}}\Big)^{\frac{1}{p-1}}\frac{\mathrm{d} t}{t}\nonumber\\
    &\Big.+\int_0^{\infty}\Big(\frac{\int_{B(x,t)}(\mathbf{W}_{\alpha,p}\sigma)^{\gamma_{2}q_1}\mathrm{d}\sigma}{t^{n-\alpha p}}\Big)^{\frac{1}{p-1}}\frac{\mathrm{d} t}{t}\nonumber\\
    &\Big.+\mathbf{W}_{\alpha,p}\mu\Big] \nonumber\\
    &={\lambda}^{\frac{q_1}{p-1}}C[I+II+III+IV +\wa\mu], 
\end{align}
It's enough to estimate $III$ and $IV$, since $I,II$ are of type $III$.

By lemma \ref{lemD} and \cite[Proof of thm.~1.1]{bib1} we have:
 \begin{align}
     III&\leq \int_0^{\infty}\Big(\frac{c\sigma(B(x,2t))+c\sigma(B(x,t))}{t^{n-\alpha p}}\Big)^{\frac{1}{p-1}}\frac{\mathrm{d} t}{t} \nonumber\\
     & \leq c\int_0^{\infty}\Big(\frac{\sigma(B(x,2t))}{t^{n-\alpha p}}\Big)^{\frac{1}{p-1}}\frac{\mathrm{d} t}{t} +c \, \mathbf{W}_{\alpha,p}\sigma(x) \nonumber\\
     & \leq c\, 2^{\frac{n-\alpha p}{p-1}}\int_0^{\infty}\Big(\frac{\sigma(B(x,t))}{t^{n-\alpha p}}\Big)^{\frac{1}{p-1}}\frac{\mathrm{d} t}{t} +c \, \mathbf{W}_{\alpha,p}\sigma(x)\nonumber\\
 & \leq c \mathbf{W}_{\alpha,p}\sigma(x),
 \end{align}
 and 
\begin{equation*}
   IV\leq c\Big[\mathbf{W}_{\alpha,p}\sigma(x)+(\mathbf{W}_{\alpha,p}\sigma(x))^{\frac{q_1}{p-1}{\gamma_2}+1}\Big]. 
\end{equation*}
 Combining everything together we have:
\begin{equation*}
    \wa(\overline{v}^{q_1}\mathrm{d}\sigma+\mathrm{d}\mu)(x)\leq {\lambda}^{\frac{q_1}{p-1}}C[\wa\nu(x)+\mathbf{W}_{\alpha,p}\sigma(x)+(\mathbf{W}_{\alpha,p}\sigma(x))^{\frac{q_1}{p-1}{\gamma_2}+1} +\wa\mu(x)]
\end{equation*}
Likewise, by symmetry we also have
\begin{equation*}
    \wa(\overline{u}^{q_2}\mathrm{d}\sigma+\mathrm{d}\nu)(x)\leq {\lambda}^{\frac{q_2}{p-1}}C[\wa\mu(x)+\mathbf{W}_{\alpha,p}\sigma(x)+(\mathbf{W}_{\alpha,p}\sigma(x))^{\frac{q_2}{p-1}{\gamma_1}+1} +\wa\nu(x)]
\end{equation*}
By choosing $\lambda>0$ large enough we guarantee that $(\overline{u},\overline{v})$ is a supersolution. Using standard iteration arguments, we conclude that there is a solution $(u,v)$ to system \eqref{IW} satisfying 
 \begin{align*}
     &\underline{u}\leq u\leq \overline{u}\\
     &\underline{v}\leq v\leq \overline{v}\\
 \end{align*}
 Since by \cite[Lem.~3.4]{bib1}, $\overline{u},\overline{v}\in L_{\mathrm{loc}}^{s}(\mathds{R}^n,\, \mathrm{d} \mu)+L_{\mathrm{loc}}^{s}(\mathds{R}^n,\, \mathrm{d} \nu)+L_{\mathrm{loc}}^{s}(\mathds{R}^n,\, \mathrm{d} \sigma)$, we automatically have $u,v\in L_{\mathrm{loc}}^{s}(\mathds{R}^n,\, \mathrm{d} \mu)+L_{\mathrm{loc}}^{s}(\mathds{R}^n,\, \mathrm{d} \nu)+L_{\mathrm{loc}}^{s}(\mathds{R}^n,\, \mathrm{d} \sigma)$
\end{proof}
With theorem \ref{gen} proved, one can easily prove existence results to systems of the form \eqref{inS}. Also, the following systems:
\begin{equation}
     \left\{
\begin{aligned}
& -\Delta_p u= \sigma\, v^{q_1} + \mu, & u>0 \quad \mbox{in}\quad \mathds{R}^n,&\\ 
& -\Delta_p v= \sigma\, u^{q_2} + \nu, & v>0 \quad \mbox{in}\quad \mathds{R}^n,&\\
& \varliminf_{|x|\to \infty} u(x)=0, & \varliminf_{|x|\to \infty} v(x)=0, &
\end{aligned}
\right.
 \end{equation}
and
 \begin{equation}
     \left\{
\begin{aligned}
& (-\Delta)^{\alpha} u= \sigma\, v^{q_1} + \mu, & u>0 \quad \mbox{in}\quad \mathds{R}^n,&\\ 
& (-\Delta)^{\alpha} v= \sigma\, u^{q_2} + \nu, & v>0 \quad \mbox{in}\quad \mathds{R}^n,&\\
& \varliminf_{|x|\to \infty} u(x)=0, & \varliminf_{|x|\to \infty} v(x)=0.&
\end{aligned}
\right.
 \end{equation}
can be analyzed in a very similar way, since they are both connected to \eqref{IW}.
\section{Final remarks and open questions}\label{last}
The results above lead to these natural questions:
\begin{enumerate}
	\item In these notes we have assumed $k<\frac{n}{2}$ because the Wolff potential $\wa\sigma$ estimates requires $0<\alpha<\frac{n}{p}$, which in our case is the same thing of requiring $k<\frac{n}{2}$. We may then ask what happens if $k>\frac{n}{2}$, namely, do we still have any sort of Brezis-Kamin estimates \eqref{est1} in case we have non trivial solutions? A particular case would be the system below not covered in this article: 
\begin{equation}
     \left\{
\begin{aligned}
& \det[D^{2}u]= \sigma\, v^{q_1}, \quad v>0 \quad \mbox{in}\quad \mathbb{R}^n,\\ 
& \det[D^{2}v]= \sigma\, u^{q_2}, \quad u>0 \quad \mbox{in}\quad \mathbb{R}^n,\\
& \liminf\limits_{|x|\to \infty}u(x)=0, \quad \liminf\limits_{|x|\to \infty}v(x)=0,\\
& \text{u and v convex.}
\end{aligned}
\right.
 \end{equation}
Labutin's local estimates \cite{labutin} are still valid, so it may be possible to still obtain solutions satisfying  \eqref{est1} but using different methods of the ones presented here.
\item We haven't discuss here the regularity of the solutions, but we believe some regularity theory could be stablished based on the regularity theory which already exists for equations involving the k-Hessian. For example, by using the same type of arguments as the ones presented in \cite{wang2009} and \cite{chou2001}. In particular, it's possible but not clear if Holder regularity could a priori be obtained for \eqref{FNS}, depending on the values of $k, n, q_{1}, q_{2}.$
\end{enumerate}
\bibliography{sn-article}


\begin{thebibliography}{12}
\ifx \bisbn   \undefined \def \bisbn  #1{ISBN #1}\fi
\ifx \binits  \undefined \def \binits#1{#1}\fi
\ifx \bauthor  \undefined \def \bauthor#1{#1}\fi
\ifx \batitle  \undefined \def \batitle#1{#1}\fi
\ifx \bjtitle  \undefined \def \bjtitle#1{#1}\fi
\ifx \bvolume  \undefined \def \bvolume#1{\textbf{#1}}\fi
\ifx \byear  \undefined \def \byear#1{#1}\fi
\ifx \bissue  \undefined \def \bissue#1{#1}\fi
\ifx \bfpage  \undefined \def \bfpage#1{#1}\fi
\ifx \blpage  \undefined \def \blpage #1{#1}\fi
\ifx \burl  \undefined \def \burl#1{\textsf{#1}}\fi
\ifx \doiurl  \undefined \def \doiurl#1{\url{https://doi.org/#1}}\fi
\ifx \betal  \undefined \def \betal{\textit{et al.}}\fi
\ifx \binstitute  \undefined \def \binstitute#1{#1}\fi
\ifx \binstitutionaled  \undefined \def \binstitutionaled#1{#1}\fi
\ifx \bctitle  \undefined \def \bctitle#1{#1}\fi
\ifx \beditor  \undefined \def \beditor#1{#1}\fi
\ifx \bpublisher  \undefined \def \bpublisher#1{#1}\fi
\ifx \bbtitle  \undefined \def \bbtitle#1{#1}\fi
\ifx \bedition  \undefined \def \bedition#1{#1}\fi
\ifx \bseriesno  \undefined \def \bseriesno#1{#1}\fi
\ifx \blocation  \undefined \def \blocation#1{#1}\fi
\ifx \bsertitle  \undefined \def \bsertitle#1{#1}\fi
\ifx \bsnm \undefined \def \bsnm#1{#1}\fi
\ifx \bsuffix \undefined \def \bsuffix#1{#1}\fi
\ifx \bparticle \undefined \def \bparticle#1{#1}\fi
\ifx \barticle \undefined \def \barticle#1{#1}\fi
\bibcommenthead
\ifx \bconfdate \undefined \def \bconfdate #1{#1}\fi
\ifx \botherref \undefined \def \botherref #1{#1}\fi
\ifx \url \undefined \def \url#1{\textsf{#1}}\fi
\ifx \bchapter \undefined \def \bchapter#1{#1}\fi
\ifx \bbook \undefined \def \bbook#1{#1}\fi
\ifx \bcomment \undefined \def \bcomment#1{#1}\fi
\ifx \oauthor \undefined \def \oauthor#1{#1}\fi
\ifx \citeauthoryear \undefined \def \citeauthoryear#1{#1}\fi
\ifx \endbibitem  \undefined \def \endbibitem {}\fi
\ifx \bconflocation  \undefined \def \bconflocation#1{#1}\fi
\ifx \arxivurl  \undefined \def \arxivurl#1{\textsf{#1}}\fi
\csname PreBibitemsHook\endcsname

\bibitem[\protect\citeauthoryear{Brezis and Kamin}{1992}]{brezis92}
\begin{barticle}
\bauthor{\bsnm{Brezis}, \binits{H.}},
\bauthor{\bsnm{Kamin}, \binits{S.}}:
\batitle{Sublinear elliptic equations in rn}.
\bjtitle{manuscripta mathematica}
\bvolume{74}(\bissue{1}),
\bfpage{87}--\blpage{106}
(\byear{1992})
\doiurl{10.1007/BF02567660}
\end{barticle}
\endbibitem

\bibitem[\protect\citeauthoryear{Cao and Verbitsky}{2015}]{verb15}
\begin{barticle}
\bauthor{\bsnm{Cao}, \binits{D.}},
\bauthor{\bsnm{Verbitsky}, \binits{I.}}:
\batitle{Finite energy solutions of quasilinear elliptic equations with
  sub-natural growth terms}.
\bjtitle{Calculus of Variations and Partial Differential Equations}
\bvolume{52}(\bissue{3}),
\bfpage{529}--\blpage{546}
(\byear{2015})
\doiurl{10.1007/s00526-014-0722-0}
\end{barticle}
\endbibitem

\bibitem[\protect\citeauthoryear{Cao and Verbitsky}{2016}]{verb16}
\begin{barticle}
\bauthor{\bsnm{Cao}, \binits{D.T.}},
\bauthor{\bsnm{Verbitsky}, \binits{I.E.}}:
\batitle{Pointwise estimates of brezis–kamin type for solutions of sublinear
  elliptic equations}.
\bjtitle{Nonlinear Analysis}
\bvolume{146},
\bfpage{1}--\blpage{19}
(\byear{2016})
\doiurl{10.1016/j.na.2016.08.008}
\end{barticle}
\endbibitem

\bibitem[\protect\citeauthoryear{Cao and Verbitsky}{2017}]{verb17}
\begin{barticle}
\bauthor{\bsnm{Cao}, \binits{D.}},
\bauthor{\bsnm{Verbitsky}, \binits{I.}}:
\batitle{Nonlinear elliptic equations and intrinsic potentials of wolff type}.
\bjtitle{Journal of Functional Analysis}
\bvolume{272}(\bissue{1}),
\bfpage{112}--\blpage{165}
(\byear{2017})
\doiurl{10.1016/j.jfa.2016.10.010}
\end{barticle}
\endbibitem

\bibitem[\protect\citeauthoryear{Chou and Wang}{2001}]{chou2001}
\begin{barticle}
\bauthor{\bsnm{Chou}, \binits{K.-S.}},
\bauthor{\bsnm{Wang}, \binits{X.-J.}}:
\batitle{A variational theory of the hessian equation}.
\bjtitle{Communications on Pure and Applied Mathematics}
\bvolume{54}(\bissue{9}),
\bfpage{1029}--\blpage{1064}
(\byear{2001})
\doiurl{10.1002/cpa.1016}
\end{barticle}
\endbibitem

\bibitem[\protect\citeauthoryear{{da Silva} and {do O}}{2024}]{bib1}
\begin{barticle}
\bauthor{\bsnm{{da Silva}}, \binits{E.L.}},
\bauthor{\bsnm{{do O}}, \binits{J.M.}}:
\batitle{Quasilinear lane–emden type systems with sub-natural growth terms}.
\bjtitle{Nonlinear Analysis}
\bvolume{242},
\bfpage{113516}
(\byear{2024})
\doiurl{10.1016/j.na.2024.113516}
\end{barticle}
\endbibitem

\bibitem[\protect\citeauthoryear{da~Silva}{2024}]{dasilva24}
\begin{botherref}
\oauthor{\bsnm{Silva}, \binits{G.}}:
Radially symmetric solutions to a lane-emden type system
(2024)
{\href{https://arxiv.org/abs/2403.14030}{{arXiv:2403.14030}}}
{[math.AP]}
\end{botherref}
\endbibitem

\bibitem[\protect\citeauthoryear{Labutin}{2002}]{labutin}
\begin{barticle}
\bauthor{\bsnm{Labutin}, \binits{D.A.}}:
\batitle{Potential estimates for a class of fully nonlinear elliptic
  equations}.
\bjtitle{Duke Mathematical Journal}
\bvolume{111}(\bissue{1}),
\bfpage{1}--\blpage{49}
(\byear{2002})
\doiurl{10.1215/S0012-7094-02-11111-9}
\end{barticle}
\endbibitem

\bibitem[\protect\citeauthoryear{Phuc and Verbitsky}{2008}]{verb2008}
\begin{barticle}
\bauthor{\bsnm{Phuc}, \binits{N.C.}},
\bauthor{\bsnm{Verbitsky}, \binits{I.E.}}:
\batitle{Quasilinear and hessian equations of lane-emden type}.
\bjtitle{Annals of Mathematics}
\bvolume{168}(\bissue{3}),
\bfpage{859}--\blpage{914}
(\byear{2008})
\end{barticle}
\endbibitem

\bibitem[\protect\citeauthoryear{Trudinger and Wang}{1999}]{wang1999}
\begin{barticle}
\bauthor{\bsnm{Trudinger}, \binits{N.S.}},
\bauthor{\bsnm{Wang}, \binits{X.-J.}}:
\batitle{Hessian measures 2}.
\bjtitle{Annals of Mathematics}
\bvolume{150}(\bissue{2}),
\bfpage{579}--\blpage{604}
(\byear{1999})
\end{barticle}
\endbibitem

\bibitem[\protect\citeauthoryear{Trudinger and Wang}{2002}]{wang2002}
\begin{barticle}
\bauthor{\bsnm{Trudinger}, \binits{N.S.}},
\bauthor{\bsnm{Wang}, \binits{X.-J.}}:
\batitle{Hessian measures 3}.
\bjtitle{Journal of Functional Analysis}
\bvolume{193}(\bissue{1}),
\bfpage{1}--\blpage{23}
(\byear{2002})
\doiurl{10.1006/jfan.2001.3925}
\end{barticle}
\endbibitem

\bibitem[\protect\citeauthoryear{Wang}{2009}]{wang2009}
\begin{bbook}
\bauthor{\bsnm{Wang}, \binits{X.-J.}}:
In: \beditor{\bsnm{Chang}, \binits{S.-Y.A.}},
\beditor{\bsnm{Ambrosetti}, \binits{A.}},
\beditor{\bsnm{Malchiodi}, \binits{A.}} (eds.)
\bbtitle{The k-Hessian Equation},
pp. \bfpage{177}--\blpage{252}.
\bpublisher{Springer},
\blocation{Berlin, Heidelberg}
(\byear{2009}).
\doiurl{10.1007/978-3-642-01674-5_5}
\end{bbook}
\endbibitem

\end{thebibliography}
\end{document}